\documentclass{amsart}

\usepackage{graphicx}
\usepackage{pstool}
\usepackage{hyperref}
\usepackage{xypic}
\usepackage{epsfig}
\usepackage{amsmath, amssymb, amscd}
\usepackage{tikz}

\newtheorem{thm}{Theorem}[section]

\newtheorem{cor}[thm]{Corollary}
\newtheorem{lem}[thm]{Lemma}

\theoremstyle{definition}
\newtheorem{defn}[thm]{Definition}
\theoremstyle{remark}
\newtheorem{rem}[thm]{Remark}

\numberwithin{equation}{section}
\begin{document}

\author{Jongil Park}
\address{Department of Mathematical Sciences, Seoul National University, Seoul 151-747  \& Korea Institute for Advanced Study, Seoul 130-722, Republic of Korea}
\email{jipark@snu.ac.kr}
\urladdr{}

\author{ Ki-Heon Yun}
\address{Department of Mathematics, Sungshin Women's University, Seoul 136-742, Republic of Korea}
\email{kyun@sungshin.ac.kr }
\urladdr{}

\title[Families of nondiffeomorphic $4$-manifolds]{Families of nondiffeomorphic $4$-manifolds with the same Seiberg-Witten invariants}


\begin{abstract}   % type your abstract below
In this article, we show that, at least for non-simply connected case, there exist an infinite family of nondiffeomorphic symplectic $4$-manifolds with the same Seiberg-Witten invariants. The main techniques are knot surgery and a covering method developed in Fintushel and Stern's paper~\cite{FS:99}.
\end{abstract}

\maketitle

%%%%%%%%%%%%%%%%%%%%   Start of main body of article

\section{Introduction}

 Since the inception of gauge theory, in particular Seiberg-Witten theory, one of the fundamental problems is to find a pair of nondiffeomorphic $4$-manifolds which have the same Seiberg-Witten invariants. Despite of the fact that Seiberg-Witten theory has been very successful in the study of smooth $4$-manifolds, it is still mysterious to find a pair of simply connected nondiffeomorphic $4$-manifolds with the same Seiberg-Witten invariants. One of the main reasons is that we do not have enough techniques to distinguish smooth $4$-manifolds. But the situation is a little bit different in non-simply connected case. For example, Fintushel and Stern constructed a single pair of (non-simply connected) nondiffeomorphic symplectic $4$-manifolds with the same Seiberg-Witten invariants using knot surgery and a covering method~\cite{FS:99}.

Knot surgery, introduced by Fintushel and Stern~\cite{FS:98}, is a very effective method to construct exotic $4$-manifolds because one can modify the Seiberg-Witten invariants without changing the topological type of a given $4$-manifold. Numerous new exotic smooth $4$-manifolds were constructed by using a knot surgery technique. Note that the Seiberg-Witten invariant of a knot surgery $4$-manifold is closely related to the Alexander polynomial of the corresponding knot~\cite{FS:98} and there are infinitely many inequivalent knots with the same Alexander polynomial~\cite{Morton:78}. So it is natural to expect that the Seiberg-Witten invariant alone is not sufficient to distinguish all smooth $4$-manifolds.

In this article we extend Fintushel and Stern's result~\cite{FS:99} using the same technique. That is, we show that there are infinitely many pairs of symplectic $4$-manifolds which share the same Seiberg-Witten invariants but they are not diffeomorphic to each other. Moreover, we get the following theorem:

\smallskip

\begin{thm}~\label{Main Theorem}
For each integer $n > 0$, there exist $n$ distinct (non-simply connected) symplectic $4$-manifolds with the same Seiberg-Witten invariants which are mutually nondiffeomorphic, but homeomorphic.
\end{thm}

In order to construct such families, we first investigate a family of $2$-bridge knots with the same Alexander polynomial which Kanenobu studied~\cite{Kanenobu:89}. Note that each $2$-bridge knot has a dihedral covering link and its covering linkage invariants are well known~\cite{Burde:88}\cite{BZ:03}. So we can apply Fintushel and Stern's covering method to this family of knots. The hard part is to prove that the Seiberg-Witten invariants of the covering link surgery $4$-manifolds are different. The problem is related to distinguish the multivariable Alexander polynomials of the corresponding dihedral covering links. To do this, we compute the covering linkage invariants and we notice that the change of diagonal elements of the linking matrix has some pattern. And then we show that the corresponding multivariable Alexander polynomials of the dihedral covering links are mutually distinct by using this pattern of the change of diagonal element and the Torres' condition of a multivariable Alexander polynomial. So we conclude that the corresponding Seiberg-Witten invariants of the covering link surgery $4$-manifolds are mutually different.

\subsection*{Acknowledgements}
Jongil Park was supported by the National Research Foundation of Korea (NRF) Grant funded by the Korean Government (2008-0093866 and 2010-0019516). He also holds a joint appointment at KIAS and in the Research Institute of Mathematics, SNU.
Ki-Heon Yun was supported by the National Research Foundation of Korea(NRF) Grant funded by the Korean government (2009-0066328 and 2012R1A1B4003427).

\section{Preliminaries}

\subsection{A knot surgery $4$-manifold}

Suppose that $X$ is a simply connected oriented smooth $4$-manifold with $b^+ >1$ which contains a smoothly embedded essential torus $T$ with self-intersection $0$ and $\pi_1(X\setminus T) = 1$. We also assume that $T$ is contained in a cusp neighborhood. Let $K$ be a smooth knot in $S^3$ and let $M_K$ be a $3$-manifold obtained by performing $0$-framed surgery along a knot $K$. Let $m$ be a meridian loop of $K$ and let $T_m = S^1 \times m$ be an embedded torus in $S^1 \times M_K$. Then a knot surgery $4$-manifold $X_K$ is defined by
\[
X_K = X \sharp_{T=T_m} S^1 \times M_K = [X\setminus N(T)] \cup [S^1 \times (S^3\setminus N(K)),]
\]
where $N(T) \cong D^2 \times T$ is a tubular neighborhood of $T$ in $X$ and $N(K)$ is a tubular neighborhood of $K$ in $S^3$. Here the fiber sum operation identifies a longitude circle of $K$ and a normal circle to $T$ in $X$.

\smallskip

\begin{thm}[\cite{FS:98}]
\label{thm:FS-knot}
Suppose that $X$ is a smooth $4$-manifold which contains a $c$-embedded torus $T$ and $\pi_1(X) = 1 = \pi_1(X\setminus T)$ . Then $X_K$ is homeomorphic to $X$ and
\begin{enumerate}
 \item if  $b^+(X) > 1$ , then \[ \mathcal{SW}_{X_K} = \mathcal{SW}_X\cdot \Delta_K(t)\]
 \item if $b^+(X) = 1$, then the $[T]^\perp$-restricted Seiberg-Witten invariants of $X_K$ are
\[\mathcal{SW}_{X_K, T}^{\pm} = \mathcal{SW}_{X, T}^{\pm} \cdot \Delta_K(t)\]
\end{enumerate}
where $t=\exp(2[T])$ and $\Delta_K$ is the symmetrized Alexander polynomial of $K$.
\end{thm}

\begin{rem}
If $K$ is a fibered knot in $S^3$ and $X$ is a symplectic $4$-manifold, then $X_K$ also admits a symplectic structure. Moreover,
\(
M_K = S^1 \times_\varphi \Sigma
\) for some closed surface $\Sigma$ and a diffeomorphism $\varphi: \Sigma \to \Sigma$ and $\Delta_K(t) = \text{det}(\varphi_* - t  \text{I})$.
\end{rem}

\begin{thm}[\cite{IP:99}, \cite{FS:99}]
\label{thm:IP}
 Suppose that $X$ is a symplectic $4$-manifold with $b^+(X) > 1$ and $T$ is a symplectically embedded torus with self-intersection $0$ in a cusp neighborhood in $X$. Let $\Sigma$ be also a symplectically embedded surface with a symplectomorphism $\varphi: \Sigma \to \Sigma$ which has a fixed point$\varphi(x_0) = x_0$. Let $m_0 = S^1 \times_\varphi \{x_0\}$ and $T_0 = S^1 \times m_0 \subset S^1 \times (S^1 \times_\varphi \Sigma)$. Then $X_\varphi = X \#_{T=T_0} S^1\times (S^1 \times_\varphi \Sigma)$ is also a symplectic $4$-manifold whose Seiberg-Witten invariant is given by
\[\mathcal{SW}_{X_\varphi} = \mathcal{SW}_X\cdot \Delta(t),\]
where $t= \exp(2[T])$ and $\Delta(t)$ is the symmetrization of $\text{det} (\varphi_* -t \text{I})$.
\end{thm}

Furthermore, one can extend a knot surgery technique to link surgery as follows: If  $L = L_1 \cup L_2 \cup \cdots \cup L_p$ is an oriented $p$-component link in $S^3$ and $(X_i, T_i)$ is a pair of simply connected smooth $4$-manifold $X_i$ and a smoothly embedded torus $T_i$ with self-intersection $0$ in $X_i$ ($i=1,2, \cdots p$), then one can define a link surgery $4$-manifold by
\[
X(X_1, X_2, \cdots, X_p; L) = [S^1 \times (S^3 \setminus N(L))] \cup  \bigcup_{i=1}^p [X_i \setminus N(T_i)],
\]
where $S^1 \times \partial N(L_i)$ is identified with $\partial N(T_i)$ so that, for each $i=1,2, \cdots p$,
\[
[T_m]= [T_i] \text{ \ \ and \ \ } [\gamma_i] = [\text{pt} \times \partial D^2].
\]
Here $\gamma_i = \ell_i + \alpha_L(\ell_i)m_i$ with a meridian $m_i$ and a longitude $\ell_i$ of the component $L_i$ and $\alpha_L : \pi_1(S^3 \setminus L) \to \mathbb{Z}$ is a homomorphism defined by  $\alpha_L(m_i) =1$.

\smallskip

\begin{thm}[\cite{FS:98}]
Suppose that $X_i$ is a simply connected smooth $4$-manifold with a homologically essential torus $T_i$ in a cusp neighborhood and $\pi_1(X_i \setminus T_i) =1$.
Then $X(X_1, X_2, \cdots, X_p; L) $is a simply connected smooth $4$-manifold and its Seiberg-Witten invariant is given by
\[
\mathcal{SW}_{X(X_1, X_2, \cdots, X_p; L) } = \Delta_L (t_1, t_2, \cdots, t_p) \prod_{i=1}^p \mathcal{SW}_{E(1) \sharp_{F=T_i} X_i}
\]
where $t_i = \exp(2[T_i])$ and $\Delta_L(t_1, t_2, \cdots t_p)$ is the symmetrized multivariable Alexander polynomial of $L$, and $F$ is a generic fiber of $E(1)$.
\end{thm}

\smallskip

\subsection{A $2$-bridge knot}

Assume that $p$ and $q$ are relatively prime integers with $p$ odd.
Let us consider a $2$-bridge knot $b(p,q)$ which is defined as follows:

\begin{defn}[\cite{BZ:03}]
\label{Definition:2bridge}
A \emph{$2$-bridge knot} $b(p,q)$ is of the form
\[ C(n_1, -n_2, n_3, -n_4, \cdots, (-1)^{k-1} n_k)\]
as in Figure~\ref{fig:2bridge}, where
\[
\frac{q}{p} = \frac{1}{n_1+\frac{1}{n_2+\frac{1}{\ddots
\frac{1}{n_{k-1} + \frac{1}{n_k}}}}} =[n_1, n_2, \cdots, n_k].
\]

\begin{figure}[htb]
 \begin{center}
\includegraphics[scale=.9]{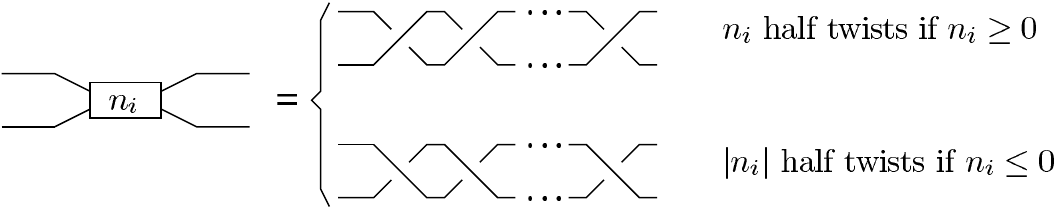}
 \includegraphics[scale=.9]{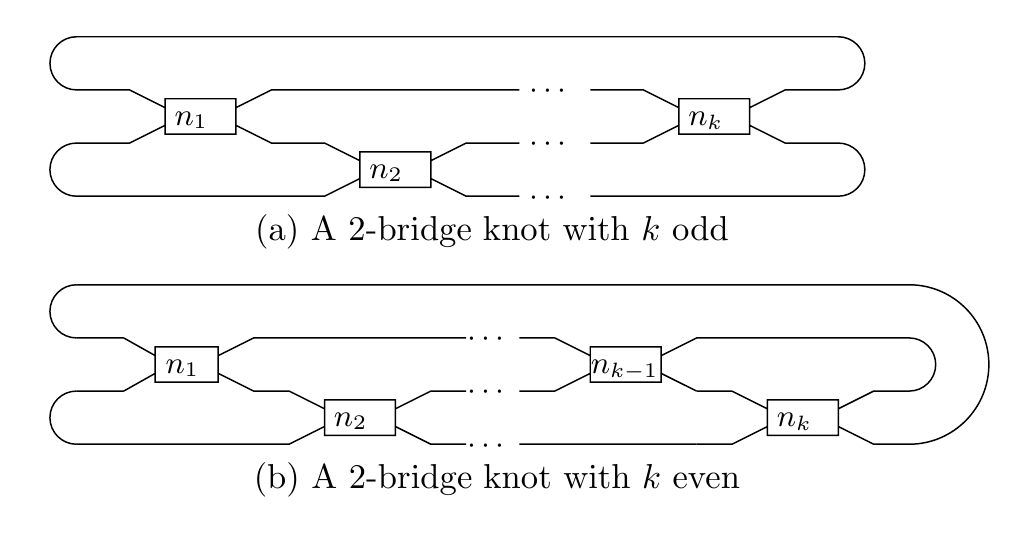}
\end{center}
\caption{A $2$-bridge knot  $C(n_1, n_2, \cdots, n_k)$}\label{fig:2bridge}
\end{figure}

It is a $4$-plat whose defining braid is
\[
\sigma_2^{n_1} \sigma_1^{-n_2} \sigma_2^{n_3} \sigma_1^{-n_4} \cdots  \sigma_1^{-n_k}  \ \ \  {\textrm{if $k$ is\, even}}, \]
\[\sigma_2^{n_1} \sigma_1^{-n_2} \sigma_2^{n_3} \sigma_1^{-n_4} \cdots  \sigma_2^{n_k}  \ \ \ {\textrm{if $k$ is\, odd}}.
\]
 Here $\sigma_i$ is a standard braid generator as in Figure 10.3 of~\cite{BZ:03}. We now denote
\[ D(n_1, n_2, \cdots, n_k) = C(2n_1, 2n_2, \cdots, 2n_k). \]
\end{defn}

\smallskip

\begin{lem}[\cite{BZ:03}, \cite{Kanenobu:89}]
\label{Lemma:2-bridge}
\begin{itemize}
\item[(a)] Two $2$-bridge knots $b(p,q)$ and $b(p',q')$ are equivalent if and only if $p=p'$ and $q=q'$ $($or $qq' \equiv 1 )$ $(\! \! \! \mod p)$.
\item[(b)] $D(n_1, n_2, \cdots, n_{2k})$ is a fibered knot if and only if each $n_i$ is $1$ or $-1$.
\item[(c)] Two $2$-bridge knots $D(a_1, a_2, \cdots, a_{2k})$ and $D(b_1, b_2, \cdots, b_{2l})$ are ambient isotopic  if and only if $k=l$ and $a_i =b_i$ or $a_i = b_{2k+1-i}$ for all $1\le i \le 2k$.
\end{itemize}
\end{lem}

\smallskip

\subsection{A Fintushel and Stern's covering method}

 We begin by mentioning that a $2$-bridge knot $b(p,q)$ is characterized by its branched double covering which is the lens space $L(p,q)$. Note that the lens space $L(p,q)$ has a universal cyclic $p$-fold covering, say, $\theta_p :S^3 \to L(p,q)$. If we denote a $2$-fold covering map branched over $b(p,q)$ by $\pi_2: L(p,q) \to S^3$, then we have the following diagram

\[
\xymatrix{
 & ( S^3, \widehat{b(p,q)})\ar[dl]_{f_2}
 \ar[dr]^{\theta_p} &    \\
 (S^3, \widetilde{b(p,q )}) \ar[dr]_{g_p}&  &
 (L(p,q), \overline{b(p,q)})\ar[dl]^{\pi_2} \\
  & (S^3, b(p,q)) &
}
\]

Here we denote $\overline{b(p,q)} = \pi_2^{-1}(b(p,q))$ and $\widehat{b(p,q)} = (\theta_p \circ \pi_2)^{-1}(b(p,q))$, called a dihedral covering link,  which is a $p$-component link in $S^3$. $g_p$ is a branched irregular $p$-fold covering and $f_2$ is a  $2$-fold covering which is branched over an unknot \cite{Cappell-Shaneson:1982}.

If $b(p,q)$ is a fibered $2$-bridge knot, we know that a $3$-manifold  $M_{b(p,q)}$, obtained by $0$-framed surgery along $b(p,q)$ in $S^3$,
is the same as  $S^1 \times_\varphi \Sigma$ for some closed surface $\Sigma$
and a homeomorphism $\varphi:\Sigma \to \Sigma$, and
\[
\overline{M}_{b(p,q)} = S^1 \times_{\varphi^2} \Sigma \to M_{b(p,q)} = S^1 \times_\varphi \Sigma
\]
is a $2$-fold covering map and there is a torus $T_{\overline{m}}= S^1 \times \pi_2^{-1}(m)$ in $S^1 \times \overline{M}_{b(p,q)}$.
Here $m$ is a meridian of $b(p,q)$.

Suppose that $X$ is a K3 surface and $F$ is a generic fiber of an elliptic fibration on $X$. Let us define a symplectic $4$-manifold by
\[
X_{b(p,q)} = X \sharp_{F = T_{\overline{m}}} S^1 \times \overline{M}_{b(p,q)},
\]
where the gluing map of the fiber sum is chosen so that the boundary of a normal disk to $F$ is matched with the lift $\overline{l}$ of a longitude to $K(p,q)$.
And let us consider its universal $p$-fold covering
\[
\widehat{\theta}_p : \widehat{X}_{b(p,q)} = X(X, \cdots, X; \widehat{b(p,q)}) \to X_{b(p,q)}
\]
which is corresponding to the $p$-fold covering
$\theta_p : (S^3, \widehat{b(p,q)}) \to (L(p,q), \overline{b(p,q)})$.

\begin{thm}[\cite{FS:99}]
\label{Theorem:Covering-Method}
Suppose that $X$ is a K3 surface and $F$ is a generic fiber of an elliptic fibration on $X$. Then
\begin{itemize}
\item[(a)] $X_{b(p,q)}$ is homeomorphic to a rational homology K3 surface with the fundamental group $\pi_1(X_{b(p,q)}) = \mathbb{Z}_p$ and
\[
\mathcal{SW}_{X_{b(p,q)}} = \text{det} (\phi_*^2 - \tau^2 \text{I} ) = \Delta(\tau)\Delta(-\tau)
\]
with $\tau = \exp([F])$.

\item[(b)] $\widehat{X}_{b(p,q)}$ is a simply connected symplectic $4$-manifold and
\begin{eqnarray*}
\mathcal{SW}_{\widehat{X}_{b(p,q)}} &=& \Delta_{\widehat{b(p,q)}} (t_1, t_2, \cdots, t_p) \prod_{i=1}^p \mathcal{SW}_{E(1)\sharp_{T=F_i} \text{K3}} \\
&=&  \Delta_{\widehat{b(p,q)}} (t_1, t_2, \cdots, t_p) \prod_{i=1}^p (t_i^{1/2} - t_i^{-1/2}),
\end{eqnarray*}
where $t_i = \exp(2[F_i])$ with $F_i=F$ in $X$ and $T$ is a generic fiber of $E(1)$.
\end{itemize}
\end{thm}

\medskip

\section{Main Construction}

\subsection{Fibered $2$-bridge knots with the same Alexander polynomial}

In this subsection we explain a method how to construct a family of fibered $2$-bridge knots with the same Alexander polynomial which Kanenobu studied \cite{Kanenobu:89}.

\begin{lem}
\label{Lemma:Same-Alex}
If two  $2$-bridge knots  $D(a_1, a_2, \cdots, a_{2n})$ and $D(b_1, b_2, \cdots, b_{2n})$ are inequivalent but they have the same Alexander polynomial, then the following four $2$-bridge knots are also inequivalent but they all have the same Alexander polynomial.
\begin{eqnarray*}
K(a,1) &=& D(a_1, a_2, \cdots, a_{2n}, 1, -a_{2n}, -a_{2n-1}, \cdots, -a_1, 1, a_1, a_2, \cdots, a_{2n}), \\
K(a,-1) &=& D(a_1, a_2, \cdots, a_{2n}, -1, -a_{2n}, -a_{2n-1}, \cdots, -a_1, -1, a_1, a_2, \cdots, a_{2n}), \\
K(b,1) &=& D(b_1, b_2, \cdots, b_{2n}, 1, -b_{2n}, -b_{2n-1}, \cdots, -b_1, 1, b_1, b_2, \cdots, b_{2n}), \\
K(b,-1) &=& D(b_1, b_2, \cdots, b_{2n}, -1, -b_{2n}, -b_{2n-1}, \cdots, -b_1, -1, b_1, b_2, \cdots, b_{2n}).
\end{eqnarray*}
\end{lem}

\begin{proof}
From Lemma~\ref{Lemma:2-bridge} above, it is clear that they are all inequivalent $2$-bridge knots. Note that
$D(a_1, a_2, \cdots, a_{2n}, \pm1, -a_{2n}, -a_{2n-1}, \cdots, -a_1, \pm1, a_1, a_2, \cdots, a_{2n})$ has a defining braid
\begin{equation}
\sigma_2^{2a_1} \sigma_1^{2a_2} \cdots \sigma_1^{2a_{2n}} \sigma_2^{\pm2} \sigma_1^{-2a_{2n}} \cdots \sigma_2^{-2a_1} \sigma_1^{\pm 2} \sigma_2^{2a_1} \cdots \sigma_{2n}^{2a_{2n}}. \label{Eqn:BraidWord}
\end{equation}
Let us consider two crossings corresponding to $\sigma_2^{\pm 2}$ and $\sigma_1^{\pm 2}$ in (\ref{Eqn:BraidWord}) and we apply two skein relations at these two locations. Then we have $K(a,1) = K(a,1)_{+,+}$, $K(a, -1) = K(a, -1)_{-,-}$ and
\begin{eqnarray*}
 \Delta_{K(a,1)_{+,+}}(t) & = &  (t^{1/2} - t^{-1/2})^2 \Delta_{K(a,1)_{0,0}}(t)  + (t^{1/2} - t^{-1/2})\Delta_{K(a,1)_{0, -}}(t)  \\ & & + (t^{1/2} - t^{-1/2})\Delta_{K(a,1)_{-, 0}}(t) + \Delta_{K(a,1)_{-, -}}(t),\\
  \Delta_{K(a,-1)_{-,-}}(t) & = &  (t^{1/2} - t^{-1/2})^2 \Delta_{K(a,-1)_{0,0}}(t)  - (t^{1/2} - t^{-1/2})\Delta_{K(a,-1)_{0, +}}(t)  \\ & & - (t^{1/2} - t^{-1/2})\Delta_{K(a,-1)_{+, 0}}(t) + \Delta_{K(a,-1)_{+, +}}(t).
\end{eqnarray*}
It is also clear that
\begin{eqnarray*}
K(a,1)_{0,0} &=& D(a_1, a_2, \cdots, a_{2n}) \sharp D(-a_{2n}, -a_{2n-1}, \cdots, -a_1) \sharp D(a_1, a_2, \cdots, a_{2n})  \\
&=& K(a,-1)_{0,0}, \\
K(a,1)_{-,-} &= &D(a_1, a_2, \cdots, a_{2n}, 0, -a_{2n}, -a_{2n-1}, \cdots, -a_1, 0, a_1, a_2, \cdots, a_{2n})  \\
&=& D(a_1, a_2, \cdots, a_{2n}) = K(a, -1)_{+. +},
\end{eqnarray*}
and $K(a,1)_{-,0}$, $K(a,1)_{0,-}$, $K(a,-1)_{+,0}$ and $K(a,-1)_{0, +}$ are all splitting links with $2$ components.
Therefore we have $\Delta_{K(a,1)}(t) = \Delta_{K(a,-1)}(t)$, and we also have $\Delta_{K(b,1)}(t) = \Delta_{K(b,-1)}(t)$ by the same way.

Next we show that $\Delta_{K(a,1)}(t) = \Delta_{K(b,1)}(t)$: Since $ D(-a_{2n}, -a_{2n-1}, \cdots, -a_1)$ is equivalent to the mirror image of $ D(a_1, a_2, \cdots, a_{2n})$,
\begin{eqnarray*}
\Delta_{D(a_1, a_2, \cdots, a_{2n}) \sharp D(-a_{2n}, -a_{2n-1}, \cdots, -a_1) \sharp D(a_1, a_2, \cdots, a_{2n})}(t) &=& (\Delta_{D(a_1, a_2, \cdots, a_{2n})}(t))^3 ,\\
\Delta_{D(b_1, b_2, \cdots, b_{2n}) \sharp D(-b_{2n}, -b_{2n-1}, \cdots, -b_1) \sharp D(b_1, b_2, \cdots, b_{2n})}(t) &=&  (\Delta_{D(b_1, b_2, \cdots, b_{2n})}(t))^3.
\end{eqnarray*}
Therefore $ \Delta_{D(a_1, a_2, \cdots, a_{2n})}(t) = \Delta_{D(b_1, b_2, \cdots, b_{2n})}(t)$ implies $\Delta_{K(a,1)}(t) = \Delta_{K(b,1)}(t)$.
\end{proof}

\begin{rem}
Lemma~\ref{Lemma:Same-Alex} above arose from Kanenobu's  construction in Theorem 1 of~\cite{Kanenobu:89}. He constructed arbitrarily many skein equivalent, amphicheiral, fibered $2$-bridge knots.  Notice that our notation
$D(a_1,  a_2, a_3, a_4, \cdots, a_{2n-1}, a_{2n})$
in Definition~\ref{Definition:2bridge} above is
$D(a_1, -a_2, a_3, -a_4, \cdots, a_{2n-1},  -a_{2n})$
in Kanenobu's notation.
\end{rem}

\smallskip

\subsection{Fibered $2$-bridge knots - Construction 1.}

\begin{defn}
\label{Definition:TorusKnot}
For each integer $n>0$,
we define $K_n = D(\underbrace{1,1, \cdots, 1, 1}_{2n})$ and
\[
K_n(\pm 1) = D(\underbrace{1,1, \cdots, 1, 1}_{2n} , \pm 1, \underbrace{-1,-1, \cdots, -1, -1}_{2n}, \pm 1, \underbrace{1,1, \cdots, 1, 1}_{2n}).
\]
\end{defn}

Let us consider a representation $\phi: \text{Br}_3 \to SL(2, \mathbb{Z})$ defined by
\[
\phi(\sigma_1) = \begin{pmatrix} 1 & -1\\ 0 & 1 \end{pmatrix}\ \textrm{and} \ \ \phi(\sigma_2) = \begin{pmatrix} 1 & 0 \\ 1 & 1 \end{pmatrix}.
\]
Here $\text{Br}_3$ denotes a braid group with 3 strands.
Then a $2$-bridge knot
$b(p,q) = D(a_1, a_2, \cdots, a_{2n})$  satisfies the equation
\[
 \pm \begin{pmatrix} r & q \\ s & p \end{pmatrix} = \phi(\sigma_2)^{2a_1} \phi(\sigma_1)^{2 a_2} \phi(\sigma_2)^{2a_3} \phi(\sigma_1)^{2 a_4} \cdots \phi(\sigma_1)^{2a_{2n}}
\]
for some integers $p, q, r$ and $s$ satisfying $pr - qs =1$.

\smallskip

\begin{lem}
\label{Lemma:const-1}
For each integer $n>0$, we have
\begin{eqnarray*}
K_n(\pm 1) &=& b((2n+1)(4n+1)(4n+3), 2n(4n+1)(4n+3)  \pm 2 (2n+1)) \\
&=& b( (2n+1)(4n+1)(4n+3),  -(4n+1)(4n+3) \pm 2 (2n+1) ).
\end{eqnarray*}
\end{lem}

\begin{proof}
Since
\[
\begin{pmatrix} 1 & 0 \\ 2 & 1 \end{pmatrix} \begin{pmatrix} 1 & -2 \\ 0 & 1 \end{pmatrix} =
 \begin{pmatrix} 1 & -2 \\ 2 & -3 \end{pmatrix} \ {\textrm{and}}
 \]
 \[
  \begin{pmatrix} 1 & -2 \\ 2 & -3 \end{pmatrix}^n =  \begin{pmatrix} (-1)^{n-1} (2n -1) & (-1)^n 2n \\ (-1)^{n-1} 2n  & (-1)^n (2n+1) \end{pmatrix},
  \]
$K_n = b(2n + 1, 2n) = b(2n+1, -1)$.
So, from the defining word of $K_n(\pm 1)$, we get
\[
  \begin{pmatrix} 1 & -2 \\ 2 & -3 \end{pmatrix}^n
\begin{pmatrix} 1 & 0 \\  \pm 2& 1\end{pmatrix}
  \begin{pmatrix} 1 & -2 \\ 2 & -3 \end{pmatrix}^{-n}
\begin{pmatrix} 1 & \mp 2 \\ 0 & 1\end{pmatrix}
  \begin{pmatrix} 1 & -2 \\ 2 & -3 \end{pmatrix}^n
\]
and it gives
\[
K_n(\pm 1) = b( (2n+1)(4n+1)(4n+3),  2n(4n+1) (4n+3) \pm  2(2n+1)).
\]
Since,  for each $\varepsilon \in \{ \pm 1\}$,
\begin{eqnarray*}
\lefteqn{\{2n (4n+1)(4n+3) + 2\varepsilon (2n+1) \} \{ -(4n+1)(4n+3) + 2\varepsilon (2n+1)\} } \\
&=& -2n( 4n+1)^2 (4n+3)^2 + 4\varepsilon n (2n+1)(4n+1)(4n+3) \\
& &  -2\varepsilon (2n+1)(4n+1)(4n+3) + 4 (2n+1)^2 \\
&\equiv & (4n+1)^2 (4n+3)^2  + (4n+1)(4n+3) +1  \pmod{ (2n+1)(4n+1)(4n+3) } \\
&=& 4(2n+1)^2(4n+1)^2 (4n+3)^2 + 1  \\
&\equiv& 1  \pmod{  (2n+1)(4n+1)(4n+3) },
\end{eqnarray*}
we finally get
\[ K_n(\pm 1) = b( (2n+1)(4n+1)(4n+3),  -(4n+1)(4n+3) \pm 2 (2n+1) ). \]
\end{proof}

\smallskip

\subsection{Fibered $2$-bridge knots - Construction 2.}

\begin{defn}
\label{Definition:KnotExample}
Let us define inductively a family of $2$-bridge knots as follows:
\begin{itemize}
\item[(a)] Set  $W(0,0) = 1, 1$ and $K(0,0) = D(W(0,0))$.
\item[(b)] For each integer $n>0$ and $i = \sum_{j=0}^{n-1} \varepsilon_j 2^{j}$ with $\varepsilon_j \in \{ 0, 1\}$, define a list $W(n, i)$ by
\[\ \ \ \ \
W(n\!-\!1, \! \sum_{j=0}^{n-2} \varepsilon_j 2^{j}),
(-1)^{\varepsilon_{n\!\!-\!\!1}\!+\!1}, -W(n\!-\!1, \! \sum_{j=0}^{n-2} \varepsilon_j 2^{j}), (-1)^{\varepsilon_{n\!\!-\!\!1}\!+\!1},  W(n\!-\!1, \! \sum_{j=0}^{n-2} \varepsilon_j 2^{j})
\]
and $K(n, i) = D(W(n,i))$.
\end{itemize}
\end{defn}

\begin{rem} For example, we have
\begin{eqnarray*}
K(1,0) &=&  D( 1, 1, -1, -1, -1, -1, 1, 1) \\
K(1, 1) &=& D(1, 1, 1, -1, -1, 1, 1, 1) \\
K(2, 0) &=&  D(1, 1, -1, -1, -1, -1, 1, 1, -1, -1, -1, 1, 1, 1, 1, -1, -1, -1, \\ & &1, 1, -1, -1, -1, -1, 1, 1) \\
K(2,1) &=&  D(1, 1, 1, -1, -1, 1, 1, 1, -1, -1, -1, -1, 1, 1, -1, -1, -1, -1, \\ & & 1, 1, 1, -1, -1, 1, 1, 1)\\
K(2, 2) &=&  D(1, 1, -1, -1, -1, -1, 1, 1, 1, -1, -1, 1, 1, 1, 1, -1, -1, 1, \\ & & 1, 1, -1, -1, -1, -1, 1, 1) \\
K(2,3) &=&  D(1, 1, 1, -1, -1, 1, 1, 1, 1, -1, -1, -1, 1, 1, -1, -1, -1, 1, \\ & & 1, 1, 1, -1, -1, 1, 1, 1)
\end{eqnarray*}
\end{rem}

\smallskip

\begin{lem}
\label{Lemma:Const-2}
$K(n, i) = b(p(n), q(n,i))$, where $p(n)$ and $q(n,i)$ are defined by the following recursion relation:
\begin{itemize}
\item[(a)] $p(0) = 3$ and $p(n+1) = p(n) \{ 4p(n)^2 -1\}$, for all $n\ge 0$
\item[(b)] $q(0,0) = 2$ and
\noindent
 \[q(n+1,  \sum_{j=0}^{n} \varepsilon_j 2^{j}) = q(n,  \sum_{j=0}^{n-1} \varepsilon_j 2^{j}) \{ 4p(n)^2 -1\} + 2(-1)^{\varepsilon_n + 1} p(n),\]
\noindent
for all $n\ge 0$ and $i= \sum_{j=0}^{n} \varepsilon_j 2^{j}$  with $\varepsilon_j \in \{0, 1\}$.
\end{itemize}
\end{lem}

\begin{proof}
We get an initial term $p(0)= p(0,0) = 3$ and $q(0,0) = 2$ from the relation
\[
\pm \begin{pmatrix} r(0, 0) & q(0, 0) \\ s(0, 0) & p(0, 0) \end{pmatrix} = \begin{pmatrix} 1 & 0 \\ 2 & 1 \end{pmatrix} \begin{pmatrix} 1 & -2 \\ 0 & 1 \end{pmatrix} =
 \begin{pmatrix} 1 & -2 \\ 2 & -3 \end{pmatrix} .
 \]

For a recursive relation, if $i= \sum_{j=0}^{n} \varepsilon_j 2^{j}$ and, for the simplicity of writing, let
\[
\begin{pmatrix} r(n,  \sum_{j=0}^{n-1} \varepsilon_j 2^{j}) & q(n,  \sum_{j=0}^{n-1} \varepsilon_j 2^{j}) \\ s(n,  \sum_{j=0}^{n-1} \varepsilon_j 2^{j}) & p(n,  \sum_{j=0}^{n-1} \varepsilon_j 2^{j}) \end{pmatrix}
\begin{pmatrix} 1 & 0 \\ (-1)^{\varepsilon_n +1} & 1\end{pmatrix}
= \begin{pmatrix} r & q \\ s & p \end{pmatrix},
\]
then
\begin{eqnarray*}
%\pm \lefteqn{ \begin{pmatrix} r(n+1,  i) & q(n+1,  i) \\ s(n+1,  i) & p(n+1,  i)
%\end{pmatrix} }
% & = &
\pm \begin{pmatrix} r(n\!+\!1,  i)\! &\! q(n\!+\!1,  i) \\ s(n\!+\!1,  i)\! & \!p(n\!+\!1,  i) \end{pmatrix}
\! =\! \begin{pmatrix} r\! &\! q \\ s\! &\! p \end{pmatrix}
\! \begin{pmatrix} 1\!\! &\!\! 0 \\ (-1)^{\varepsilon_n \!\!+\!\!1}2\! &\! 1\end{pmatrix}
\! \begin{pmatrix} r\! &\! q \\ s\! &\! p \end{pmatrix}^{-1}
\! \begin{pmatrix} 1\!\! &\!\! -(-1)^{\varepsilon_n \!+\!1}2 \\ 0\!\! &\!\! 1\end{pmatrix}
\! \begin{pmatrix} r\! &\! q \\ s\! &\! p \end{pmatrix}.
\end{eqnarray*}
From this, we get
\begin{eqnarray*}
\pm p(n+1, i) &\! =\! & p\cdot  \{p \cdot  r  - q \cdot  s - 4p^2\} = p\cdot  \{1- 4p^2\} \\
\pm q(n+1,i) &\! =\! & -q^2 \cdot  s - 2p^2 \cdot  r \cdot  (-1)^{\varepsilon_n +1}  -4p^2 \cdot  q + p \cdot  q \cdot  r + 2p \cdot  q \cdot  s \cdot  (-1)^{\varepsilon_n +1}  \\
 &\! =\! & q\cdot  \{p \cdot r - q \cdot  s\} -2p \cdot  (-1)^{\varepsilon_n +1 } \cdot  \{p \cdot  r -q \cdot  s\} -4p^2 \cdot  q  \\
 &\! =\! & q \cdot  \{1-4p^2\} -2 (-1)^{\varepsilon_n +1 }  \cdot  p.
\end{eqnarray*}
Hence, since $p >0$, we finally get
\begin{eqnarray*}
p(n+1) &\! =\! &  p(n) \cdot  \{ 4p(n)^2 - 1\} \ \ {\textrm{and}} \\
q(n+1,i) &\! =\! &  q(n,\sum_{j=0}^{n-1} \varepsilon_j 2^{j}) \cdot  \{4p(n)^2 - 1\}  +2(-1)^{\varepsilon_n +1 } \cdot  p(n) .
\end{eqnarray*}
\end{proof}

\begin{rem}
By a computation above, we know that  $q(n,i)$ is an even integer and
\begin{eqnarray*}
p(n+1) &=& p(0) \cdot  \prod_{k=0}^n \{ 4p(k)^2 -1\}, \\
q(n+1, \sum_{j=0}^{n} \varepsilon_j 2^{j}) &=& \frac{2}{3} p(n+1) + \sum_{j=0}^n \frac{2(-1)^{\varepsilon_j + 1}  \cdot p(n+1)}{4p(j)^2 -1} \label{Eqn:q(n,i)}.
\end{eqnarray*}
\end{rem}

\begin{lem}
\label{Lemma:q'}
Suppose that $q'(n,i)$ is an integer defined by the following recursive relation:
\begin{eqnarray*}
q'(0, 0) & = & -1, \ \ \ {\textrm{and}}  \\
q'(n+1,  \sum_{j=0}^{n} \varepsilon_j 2^{j}) & = &  q'(n,  \sum_{j=0}^{n-1} \varepsilon_j 2^{j}) \cdot  \{ 4p(n)^2 -1\} + 2(-1)^{\varepsilon_n + 1}
\cdot  p(n)
\end{eqnarray*}
for all $n\ge 0$ and $i= \sum_{j=0}^{n} \varepsilon_j 2^{j}$ with $\varepsilon_j \in \{ 0, 1\}$. Then it satisfies
\[
q(n,i) \cdot q'(n,i) \equiv 1 \pmod{ p(n) }
\]
for all $n\ge 0$ and $i = 0, 1, 2, \cdots, 2^n -1$.
\end{lem}

\begin{proof}
Let us prove it by induction: $q(0,0) \cdot q'(0,0) \equiv 1$ (mod $3$) is clear. Suppose that it is true for all $n \ge 0$ and for each $i= 0, 1,2, \cdots 2^n-1$. Let $i= \sum_{j=0}^{n} \varepsilon_j 2^{j}$ and we denote $i_{n} = \sum_{j=0}^{n-1} \varepsilon_j 2^{j}$.
Then, since  $p(n+1) = p(n) (4p(n)^2 -1)$ and $q(n, i_n)\cdot  q'(n, i_n) = k \cdot p(n) + 1$ for some integer $k$, we get
\begin{eqnarray*}
%\lefteqn{q(n+1, i) \cdot q'(n+1,i)}
 q(n+1, i) \cdot q'(n+1,i)
 &=& \{q(n, i_n) \cdot  (4 p(n)^2 -1)  + 2(-1)^{\varepsilon_n +1} \cdot  p(n)\} \\ & &  \cdot  \{q'(n, i_n) \cdot (4 p(n)^2 -1) + 2(-1)^{\varepsilon_n +1} \cdot  p(n)\}\\
&=& (k\cdot  p(n) + 1)\cdot  (4p(n)^2-1)^2  + 4p(n)^2 \\ & & + 2(-1)^{\varepsilon_n +1}  \cdot  p(n) \cdot   (4p(n)^2 -1) \cdot  \{ q(n,i_n) + q'(n, i_n)\}    \\
&=& 1 + 4p(n)^2 \cdot  (4p(n)^2 -1) +  k \cdot  p(n) \cdot  (4p(n)^2 -1)^2  \\ & & +  2(-1)^{\varepsilon_n +1} \cdot  p(n) \cdot  (4p(n)^2 -1) \cdot  \{ q(n,i_n)   + q'(n, i_n)\}   \\
&\equiv& 1 \  \pmod{ p(n+1) }.
\end{eqnarray*}
\end{proof}

\subsection{Covering linkage invariants}

Let $\widehat{b(p,q)}$ be a dihedral covering link of an oriented $2$-bridge knot $b(p,q)$ with a relatively prime pair of odd integers $(p,q)$ satisfying $-p < q < p$. Then
\[
\widehat{b(p,q)} = L_1 \cup L_2 \cup \cdots \cup L_{p}
\]
 is an oriented $p$-component link and we can give an ordering to the link components so that its linking number satisfies
\[
\text{lk}(L_r, L_s) = \begin{cases} (-1)^{[\frac{q}{p} (s-r)]}, &  r \ne s \\
 -\sum_{k \ne r} \text{lk} (L_r, L_k), & r=s  \end{cases}
\]
where $[x]$ means the greatest integer less than or equal to $x$.
It is well known that $\text{lk}(L_r, L_s) = \text{lk}(L_s, L_r)$ and $\text{lk}(L_r, L_s) = \text{lk}(L_{r-k}, L_{s-k})$  for each integer $k$, where the subindices are computed in modulo $p$~\cite{BZ:03} so that the covering linkage matrix
\[
\begin{pmatrix} \text{lk} (L_r, L_s) \end{pmatrix}
\]
is a symmetric circulant matrix~\cite{pre06032536}.

\begin{rem}
For each $2$-bridge knot $b(p,q)$, we can always choose a relatively prime pair of odd integers  $p$ and $q$ satisfying  $p>1$ and $-p< q < p$~\cite{BZ:03}. So from now on we assume this condition.
\end{rem}

\begin{defn}
For a given $2$-bridge knot $b(p,q)$,
we define a diagonal element $d(b(p,q))$ by
\[
d(b(p,q)) = -\sum_{k \ne r} \text{lk} (L_r, L_k) = - \sum_{k=1}^{p-1} (-1)^{[\frac{q}{p} k]}  = - 2 \sum_{k=1}^{(p-1)/2} (-1)^{[\frac{q}{p} k]}
\]
where $\widehat{b(p,q)} = L_1 \cup L_2 \cup \cdots \cup L_{p}$ is a dihedral covering link of $b(p,q)$.
\end{defn}

\smallskip

\begin{thm}
\label{Thm:diagonal}
Suppose that $p$ and $q$ are relatively prime odd integers satisfying $p \ge 3$ and $0 < q < p$. Let $b(p, q)$ be a $2$-bridge knot. Then
\begin{equation}\label{Eqn:diagonal}
d(b(4p^3 - p, (4p^2 -1) q  \pm  2p ) ) = d(b(p,q)) \pm 2.
\end{equation}
\end{thm}

\begin{proof}
We prove only $d(b(4p^3 - p, (4p^2 -1) q +  2p ) ) = d(b(p,q)) + 2$.
The other case, $d(b(4p^3 - p, (4p^2 -1) q -  2p ) ) = d(b(p,q)) - 2$,
is also obtained by the same way.

To do this, we first count the number of integer $k$ satisfying
$0< k < 4 p^3\! -\! p$ and
\begin{equation}
\label{Eqn:ineq-1}
(-1)^{\left[ \frac{q}{p}  k \right] } = -1
\text{ and }
(-1)^{\left[ \frac{(4p^2 -1) q + 2p}{4p^3 -p}  k \right] } = +1,
\end{equation}
or
\begin{equation}
\label{Eqn:ineq-2}
(-1)^{\left[ \frac{q}{p}  k \right] } = +1
\text{ and }
(-1)^{\left[ \frac{(4p^2 -1) q + 2p}{4p^3 -p}  k \right] } = -1.
\end{equation}
Let us divide a set $\{ k \ | \ 0 \le k < 4p^3 - p \}$ into four subsets $A_0, A_1, A_2$ and $A_3$ as follows: For each $0 \le i \le 3$, we define
\begin{equation*}
\label{Eqn:k}
A_i = \{ k \ |\ 0 \le k < 4p^3 -p \text{ satisfying }
k = m (4p^2-1) + ip^2 + lp + s
\text{ for } 0 \le m, l, s, \le p-1 \}
\end{equation*}
Then we have
\[
A_0 \cup A_1 \cup A_2 \cup A_3 = \{ k \ | \ 0 \le k < 4p^3 - p \}
 \ \text{ and }
\]
\[
A_0 \cap A_3 = \{ k \ | \   1 \le k < 4p^3 -p \ \textrm{ and }\
k \equiv 0 \! \! \pmod{\! (4p^2 -1)}  \}.
\]
Observe that, for each $0 \le i \le 3$,
\begin{equation}
\label{Eqn:n}
\frac{q}{p} \cdot (m(4p^2-1) + ip^2 + lp + s) = (4m + i)pq + lq + 
\frac{q \cdot (s-m)}{p} \ \ {\textrm{ and }}
\end{equation}
\begin{align}
& \frac{(4p^2 -1) q + 2p}{4p^3 -p} \cdot
\left(m(4p^2-1) + ip^2 + lp + s \right)  \label{Eqn:n+1} \\
& = (4m + i)pq + lq + \frac{q \cdot (s-m)}{p} + 2m  
+ \frac{2(ip^2 + lp +s)}{4p^2 -1}. \notag
\end{align}
Since there are integers $M$ and $0 \le t \le p\!-\!1$, determined uniquely by $m$ and $s$, satisfying
$q \cdot (s-m) = p \cdot M + t$, Equations (\ref{Eqn:n}) and (\ref{Eqn:n+1}) are modified as follows:
\begin{equation}
\label{Eqn:n'}
\frac{q}{p} \cdot (m(4p^2-1) + ip^2 + lp + s)  = (4m + i)pq + lq +  M  + \frac{t}{p}
\end{equation}
\begin{align}
& \frac{(4p^2 -1) q + 2p}{4p^3 -p} \cdot
\left(m(4p^2-1) + ip^2 + lp + s \right)  \label{Eqn:n+1'} \\
& = (4m + i)pq + lq + M + 2m +  \frac{t}{p} + \frac{2(ip^2 + lp +s)}{4p^2 -1}. \notag
\end{align}
Hence, for each $0 \le i \le 3$ and $0 \le t, l, s \le p-1$, we get
\[
0 \le \frac{t}{p} + \frac{2(ip^2 + lp +s)}{4p^2 -1} < 3.
\]
So, by comparing the integer parts of (\ref{Eqn:n'}) and (\ref{Eqn:n+1'}),
only when
\begin{equation}
\label{Eqn:change}
1 \le  \frac{t}{p} + \frac{2(ip^2 + lp +s)}{4p^2 -1} < 2,
\end{equation}
the corresponding integer $k = m(4p^2-1) + ip^2 + lp + s$ satisfies either (\ref{Eqn:ineq-1}) or (\ref{Eqn:ineq-2}). Otherwise, the integer $k$ satisfies
\begin{equation}
\label{Eqn:ineq-1-1}
(-1)^{\left[ \frac{q}{p}  k \right] } = 1 =
(-1)^{\left[ \frac{(4p^2 -1) q + 2p}{4p^3 -p}  k \right] }
\ {\textrm{ or }} \
(-1)^{\left[ \frac{q}{p}  k \right] } = -1 =
(-1)^{\left[ \frac{(4p^2 -1) q + 2p}{4p^3 -p}  k \right] }.
\end{equation}

Next we consider a map
\begin{equation}
\label{Eqn:map}
\phi : A_0 \cup A_1 \cup  A_2 \cup A_3  \to  A_0 \cup A_1 \cup A_2 \cup  A_3
\end{equation}
defined by
\[
\phi(k)  \equiv k + p^2 \pmod{ 4p^3 - p}.
\]
Then $\phi$ is a one-to-one map which satisfies following  properties:
\begin{itemize}
\item If $k \in A_0$ satisfies (\ref{Eqn:change}), then $\phi(k) \in A_1$ also satisfies (\ref{Eqn:change}) because
\[\ \ \ \ \ \
1 \le  \frac{t}{p} + \frac{2( lp +s)}{4p^2 -1} < \frac{t}{p} + \frac{2(p^2 +  lp +s)}{4p^2 -1} \le \frac{8p^3 -4p^2 -3p + 1}{p \cdot (4p^2 -1)} < 2
 \]
 for any $0 \le t, l, s \le p-1$.

Note that, if $k \in A_0$ satisfies (\ref{Eqn:ineq-1}), then $\phi(k) \in A_1$ satisfies  (\ref{Eqn:ineq-2}) and vice versa,  if $k \in A_0$ satisfies (\ref{Eqn:ineq-2}), then $\phi(k) \in A_1$ satisfies  (\ref{Eqn:ineq-1}).

\item If  $k \in A_i$ ($i \! =\! 1, 2$) satisfies (\ref{Eqn:change}) and  $k \ne 4p^3 - p -p^2$, then exactly one of $\phi(k) = k + p^2$ or $\phi^{-1}(k) = k-p^2$ satisfies (\ref{Eqn:change}).
We can prove this as follows: First note that
at most one of  $\phi(k) \in A_{i+1}$ or $\phi^{-1}(k) \in A_{i-1}$ satisfies (\ref{Eqn:change}).
Suppose that none of $\phi(k)$ and $\phi^{-1}(k)$ satisfies (\ref{Eqn:change}). Then, for each $j=0, 2$, we have
\begin{equation*}
  j \le  \frac{t}{p} + \frac{2((i+j-1) p^2 + lp +s)}{4p^2 -1} < j +1 .
\end{equation*}
It implies that
\begin{equation} \label{Eqn:A1}
(6-2i)p^3 -2p \le (4p^2 -1)t + 2p(lp + s) < (6-2i)p^3 -p
\end{equation}
and it is equivalent to
\begin{equation} \label{Eqn:A1-1}
 \! \! \! (3-i)p- \frac{2p \cdot (1+s)-t}{2p^2} \le 2t+l < (3-i)p- 
 \frac{p \cdot (1 + 2s)-t}{2p^2}.
\end{equation}
Note that, since $t, l$ and $s$ are all integers satisfying $0 \le t, l, s \le p-1$ ,
\[
0 < \frac{2p \cdot (1+s)-t}{2p^2} \le 1 \ {\textrm{and}} \
0 < \frac{p \cdot (1 + 2s)-t}{2p^2} < 1.
\]
Hence the inequality (\ref{Eqn:A1-1}) makes sense only when $ \frac{2p \cdot (1+s)-t}{2p^2} = 1$, which is the case $s= p-1$ and $t=0$.
Furthermore, in this case we should have $2t+l = (3-i)p-1$. Again it is true only when $l=p-1$, $i=2$, and $m=s$, which is the case $k= 4p^3 - p -p^2$. Therefore we get a contradiction.

\item If $k = 4p^3 - p -p^2 \in A_2$, then  $k$ satisfies (\ref{Eqn:ineq-2}) but neither $\phi(k)$ nor $\phi^{-1}(k)$ satisfies (\ref{Eqn:change}).

\item If $k\in A_3$ satisfies (\ref{Eqn:change}), then $\phi^{-1}(k)$ also satisfies (\ref{Eqn:change}) because
\[
1 \le \frac{t}{p}+\frac{2(2p^2+lp+s)}{4p^2-1}<
\frac{t}{p}+\frac{2(3 p^2+lp +s)}{4p^2-1} <2.
\]
\end{itemize}

Therefore, by collecting the properties of $\phi$ above, we conclude that the number of integers $k$ satisfying (\ref{Eqn:ineq-2}) in $\{ k \, |\, 0 < k < 4p^3 -p \}$ is exactly one more than those satisfying (\ref{Eqn:ineq-1}).
Finally, by combining Equations~\ref{Eqn:ineq-1}, ~\ref{Eqn:ineq-2}, ~\ref{Eqn:ineq-1-1} and the fact that
\begin{align*}
\sum_{k=1}^{4p^3 -p-1} (-1)^{\left[ \frac{q}{p} k \right] } & =
\sum_{k=1}^{p-1} (-1)^{\left[ \frac{q}{p} k \right] } +  \sum_{w=1}^{2p^2 -1} \sum_{k=0}^{2p-1} (-1)^{\left[ \frac{q}{p} (2pw-p + k)\right] }   \\
& = \sum_{k=1}^{p-1} (-1)^{\left[ \frac{q}{p} k \right] } = -d(b(p,q)),
\end{align*}
we get
\[
d(b(4p^3 - p, (4p^2 -1) q +  2p ) ) = d(b(p,q)) + 2.
\]

\end{proof}

\smallskip

\begin{cor}
\label{Lemma:diag-Torus}
For each integer $n>0$,
\[
d(K_n(\pm 1)) = d(K_n) \pm 2 =  2 n \pm 2.
\]
\end{cor}

\begin{proof}
In Lemma~\ref{Lemma:const-1} above, we get
\begin{eqnarray*}
K_n(\pm 1) &=& b((2n+1)(4n+1)(4n+3), 2n(4n+1)(4n+3)  \pm 2 (2n+1)) \\
&=& b( (2n+1)(4n+1)(4n+3),  -(4n+1)(4n+3) \pm 2 (2n+1) ).
\end{eqnarray*}

Let us consider the mirror image $K_n(\pm 1)^*$ of $K_n(\pm 1)$. 
Then it is a $2$-bridge knot of the type
$b( (2n+1)(4n+1)(4n+3),  (4n+1)(4n+3) \mp 2 (2n+1) ).$
Then, by choosing $p = 2n + 1$ and $q= 1$, we get
\begin{align*}
(2n+1)(4n+1)(4n+3) & = 4p^3 -p \ {\textrm{ and }} \\
(4n+1)(4n+3) \mp 2 (2n+1) &= (4p^2 -1)q \mp 2p.
\end{align*}
So it  satisfies all conditions of Theorem~\ref{Thm:diagonal} above
and therefore we get
\[
d(K_n(\pm 1)^*) = d(b(2n+1, 1)) \mp 2 \ {\textrm{and}}
\]
\[
d(K_n(\pm 1)) = - d(k_n(\pm 1)^*) = - d(b(2n + 1, 1))  \pm 2 = 2n \pm 2.
\]
\end{proof}

\begin{cor}
\label{Lemma:diag-ni}
For each $\varepsilon_j \in \{0, 1\}$ with $0 \le j \le n$, we have
\[
d(K(n+1, \sum_{j=0}^{n} \varepsilon_j 2^{j})) = d(K(n, \sum_{j=0}^{n-1} \varepsilon_j 2^{j}))  + 2 (-1)^{\varepsilon_n +1} \ {\textrm{and}}
\]
\[
| \{ |d(K(n, i)) | \  | \  i = 0, 1, 2, \cdots, 2^n -1\} | = \left[\frac{n+1}{2}\right]  + 1.
\]
\end{cor}

\begin{proof}
In Lemma~\ref{Lemma:Const-2} and Lemma~\ref{Lemma:q'} above, we get
\[
K(n, i) = b(p(n), q'(n, i))
\]
as a $2$-bridge knot, where
\begin{align}
p(0) = 3, \ \ q(0,0) & = -1 \ {\textrm{and}} \ \
p(n+1) = p(n) \cdot \{4 p(n)^2-1\}, \label{Eqn:recursion-p}\\
q'(n+1, \sum_{j=0}^{n} \varepsilon_j 2^{j}) &= q'(n,  \sum_{j=0}^{n-1} \varepsilon_j 2^{j}) \cdot \{4p(n)^2 -1\} + 2(-1)^{\varepsilon_n + 1} p(n) \label{Eqn:recursion-q}
\end{align}
for all $n \ge 0$ with a convention that $0 \le i = \sum_{j=0}^{n-1} \varepsilon_j 2^{j} < 2^n$.

Since $q'(n,i) <0$ for each $n\ge 0$ and $0 \le i < 2^n$, for simplicity of a computation, we consider
$K(n, i)^* = b(p(n), - q'(n, i))$, the mirror image of $K(n, i)$.
Let $\bar{q}(n, i) = - q'(n, i)$, so that $\bar{q}(n,i)$ is a positive odd integer and
\[
d(K(n+1,i)) =- \sum_{k=1}^{p(n+1)-1} (-1)^{\left[\frac{q'(n+1,i)}{p(n+1)} k\right]} = \sum_{k=1}^{p(n+1)-1} (-1)^{\left[\frac{\bar{q}(n+1,i)}{p(n+1)} k\right]}.
\]
Since $p(n+1) = p(n) \cdot \{4p(n)^2-1\}$ and
\begin{equation*}
\bar{q}(n+1,  \sum_{j=0}^{n} \varepsilon_j 2^{j})
=\bar{q}(n,  \sum_{j=0}^{n-1} \varepsilon_j 2^{j}) \cdot \{4p(n)^2 -1\}  + 2(-1)^{\varepsilon_n}p(n),
\end{equation*}
we can apply Theorem~\ref{Thm:diagonal} to get
\begin{align*}
d(K(n+1, i)) & = -d(K(n+1, i)^*) = -\{d(K(n, i)^*) + (-1)^{\varepsilon_n} 2\} \\
&= d(K(n,i)) + (-1)^{\varepsilon_n + 1} 2.
\end{align*}
Hence $d(K(0,0)) = d(b(3, -1)) =2$ implies that
\begin{eqnarray*}
 & & \{ d(K(n,i)) \ | \ i= 0, 1, 2, \cdots 2^n -1 \} = \{ 2-2n + 4j \ | \ j=0, 1, \cdots, n \} \ \ {\textrm{and}}  \\
 & & | \{ |d(K(n, i)) | \  | \  i = 0, 1, 2, \cdots, 2^n -1 \} |  = \left[\frac{n+1}{2}\right]  + 1.
\end{eqnarray*}

\end{proof}

\subsection{A multivariable Alexander polynomial computation}

Suppose that $L = L_1 \cup L_2 \cup \cdots \cup L_p$ is an ordered oriented link with $p$ components. Let
\[
\Delta_L ( t_1, t_2, \cdots, t_p) \in \mathbb{Z}[t_1^{\pm 1}, t_2^{\pm 1}, \cdots, t_p^{\pm 1}]
\]
be its multivariable Alexander polynomial and
\[
\nabla_L ( t_1, t_2, \cdots, t_p) \in \mathbb{Z}(t_1^{\pm 1}, t_2^{\pm 1}, \cdots, t_p^{\pm 1})
\]
be its Conway potential function. Then it is known to Hartley~\cite{Hartley:1983} that
\begin{equation}
\nabla_L(t_1, t_2, \cdots, t_p) = \begin{cases}
\frac{\Delta_L(t_1^2)}{t_1 - t_1^{-1}}, & \text{ if } p = 1 \\
\Delta_L(t_1^2, t_2^2, \cdots, t_p^2), & \text{ if } p \ge 2
\end{cases}
\end{equation}
if the Alexander polynomial is symmetrized.
The reduced potential function, defined by
$\overline{\nabla}(t) = (t - t^{-1}) \nabla(t, t, \cdots, t)$,
satisfies the skein relation
$\overline{\nabla}_+ (t) = \overline{\nabla}_-(t) + (t - t^{-1}) \overline{\nabla}_0(t).$
He also defined
\[
H(t) = \frac{\overline{\nabla}_L(t)}{(t-t^{-1})^{p-1}} = \frac{\nabla_L(t, t, \cdots, t)}{(t-t^{-1})^{p-2}}
\]
which is actually the Hosokawa polynomial as in ~\cite{Hosokawa:58} and it satisfies a relation:
\[
H_+(t) = H_-(t) + (t - t^{-1})^2 H_0(t).
\]

\smallskip

\begin{lem}[Torres Condition ~\cite{Hartley:1983} ]
\label{Lemma:Torres}
 Suppose that $L = L_1 \cup L_2 \cup \cdots \cup L_p$ is a link with $p$ components in $S^3$. Then
\begin{multline} \notag
\nabla_L(t_1, t_2, \cdots , t_{i-1}, 1, t_{i+1},  \cdots, t_p)   \\  = \left( \prod_{j \ne i, j=1}^p t_j^{\ell_{ij}} - \prod_{j \ne i, j=1}^p t_j^{-\ell_{ij}} \right) \nabla_{L\setminus L_i}( t_1, t_2, \cdots , t_{i-1},  t_{i+1},  \cdots, t_p).
\end{multline}
\end{lem}

For an oriented link $L = L_1 \cup L_2 \cup \cdots \cup L_p$, one can define the corresponding graph $G$ such that each vertex is related to a link component and an edge between two vertices $v_i$ and $v_j$ with weight $\ell_{ij}$ if  $\mathrm{lk}(L_i, L_j) = \ell_{ij}$.  The graph $G$ is so called the adjacent graph of the linking matrix $\left( \ell_{ij} \right)$.

\begin{lem}[\cite{Hartley:1983}, \cite{Hoste:85}]
\label{Lemma:Hoste}
 Suppose that $L$ is an oriented link with $p$ components in $S^3$. Then
\[
H_L(1) = \mathcal{L}_{ij} = (-1)^{p-1} \sum_{g\in \mathcal{T} }\overline{g},
\]
where $\mathcal{L}_{ij}: = (-1)^{i+j} \mathrm{det} (\mathcal{L}^{ij})$  and $\mathcal{L}^{ij}$ is the $(i,j)$-minor matrix of the linking matrix $\mathcal{L}$ and $\mathcal{T}$ is the set of all trees consisting of $(p-1)$ edges and $\overline{g}$ is the product of $(p-1)$ linking numbers corresponding to the tree $g$.
\end{lem}

\begin{lem}[Cayley Theorem]
\label{Lemma:Cayley}
The number of trees with $p$ labeled vertices and $(p-1)$ edges in the complete graph $C_p$ is $p^{p-2}$.
\end{lem}

\begin{lem}
\label{Lemma:Main}
Suppose that $b(p, q)$ and $b(p, q')$ are two inequivalent $2$-bridge knots with an odd integer $p>0$. Let $\widehat{b(p, q)}$, $\widehat{b(p, q')}$ be the corresponding dihedral covering links and $\left( \ell_{ij} \right)$,  $\left( \ell_{ij}' \right)$ be the corresponding linking matrices respectively. We assume that $0 \le \ell_{ii} < \ell_{ii}'$. Then
\[
\nabla_{\widehat{b(p, q)}}(t_1, t_2, \cdots, t_p) \ne  \nabla_{\widehat{b(p, q')}}(t_{\sigma(1)}, t_{\sigma(2)}, \cdots, t_{\sigma(p)})
\]
for any permutation $\sigma \in S_p$.
\end{lem}

\begin{rem}
Note that the linking matrix
$\left( \ell_{ij}= \text{lk} (L_i, L_j) \right)$ of an oriented dihedral covering link of a $2$-bridge knot is a symmetric circulant matrix. We also know that  $\ell_{ij}, \ell_{ij}' \in \{1, -1\}$ for $i \ne j$ and $\ell_{ii}$, $\ell_{ii}'$ are multiples of $2$. If we switch an orientation of $S^3$, then all linking numbers switch their signs, so that one of $\widehat{b(p, q)}$ or its mirror image has a linking matrix $\left( \ell_{ij} \right)$ with $\ell_{ii} \ge 0$.
\end{rem}

\begin{proof}[Proof of Lemma~\ref{Lemma:Main}]
Assume that $p=2n+1$ and
\begin{align}
L := \widehat{b(p, q)} = L_1 \cup L_2 \cup \cdots \cup L_p
\ \ {\textrm{ and }} \ \
L' := \widehat{b(p, q')} = L_1' \cup L_2' \cup \cdots  \cup L_p' . \notag
\end{align}
Let $d = \ell_{ii}$ and $d' = \ell_{ii}'$ denote the diagonal elements of the corresponding linking matrices.  Then,
since $| \{ j | \ell_{1j} = -1\} | = d +  | \{ j | \ell_{1j} = 1\} |  \ge d$,
we can select $r_i$ ($ 0 \le i \le d$) satisfying
$1 = r_0  <  r_1< r_2 < r_3 < \cdots <  r_d \le p$ and ${\mathrm{lk}}(L_1, L_{r_i}) =-1$ for all $1 \le i \le d$.
Note that the following equalities hold
\[
\sum_{j  \in  \{1, 2, \cdots, p\} \setminus \{r_0, r_1, r_2, \cdots, r_d\}}   \ell_{1j} =   \sum_{j=2}^p \ell_{1j}  +  d = 0.
\]
Suppose that $\check{L} = L \setminus (L_{r_0} \cup L_{r_1} \cup \cdots \cup L_{r_d})$. Then it is a link with $(p-d -1)$ components and it satisfies
\begin{multline}\label{Eqn:Torres-1}
\nabla_L (t_1, \cdots, t_p) \vert_{ t_{r_0} = t_{r_1} = \cdots = t_{r_d} =1 } \\
= \prod_{j=0}^d \left( \prod_{i \in \{1, 2, \cdots, p\} \setminus \{r_0, r_1, r_2, \cdots, r_d\}} t_i^{\ell_{r_j i}} -  \prod_{i \in \{1, 2, \cdots, p\} \setminus \{r_0, r_1, r_2, \cdots, r_d\}} t_i^{-\ell_{r_j i}} \right) \nabla_{\check{L}}
\end{multline}
and, if we put $t_i = t $ for each $i \in \{1, 2, \cdots, p\} \setminus \{r_0, r_1, \cdots, r_d\}$, Equation (\ref{Eqn:Torres-1}) becomes $0$
for all $t \in {\mathbb R}$ because
\[
  \prod_{i \in \{1, 2, \cdots, p\} \setminus \{r_0, r_1, \cdots, r_d\}} t^{\ell_{1 i}} -  \prod_{i \in \{1, 2, \cdots, p\} \setminus \{ r_0, r_1, \cdots, r_d\}} t^{-\ell_{1 i}}  = t^0 - t^{-0} = 0.
\]

On the other hand we want to show that, for any choice $\{s_0, s_1, \cdots, s_d \}$ satisfying $1 = s_0 < s_1 < \cdots < s_d \le p$, it satisfies
\begin{multline}
\label{Eqn:Torres-2}
\nabla_{L'} (t_1, \cdots, t_p) \vert_{t_{s_0} =  t_{s_1} = \cdots = t_{s_d} =1 } \\
= \prod_{j=1}^d \left( \prod_{i \in \{1, 2, \cdots, p\} \setminus \{s_0, s_1, \cdots, s_d\}} t_i^{\ell'_{s_j i}} -  \prod_{i \in \{1, 2, \cdots, p\} \setminus \{s_0, s_1, \cdots, s_d\}} t_i^{-\ell'_{s_j i}} \right) \nabla_{\check{L}'}
\ne 0
\end{multline}
in some deleted open neighborhood of $1$  after we put $t_i = t $ for each
$i \in \{1, 2, \cdots, p\} \setminus \{s_0, s_1, \cdots, s_d\}$,
where $\check{L}' = L' \setminus (L_{s_0}' \cup L_{s_1}' \cup \cdots \cup L_{s_d}')$.

To see this, first observe that
\[
 -d' -d \le \sum_{i \in \{1, 2, \cdots, p\} \setminus \{s_0, s_1, \cdots, s_d\}}  \ell'_{s_ji} \le -d' +d <0
\]
for each $j=0, 1, \cdots, d$, and it implies that
\[
\prod_{j=0}^d \left( \prod_{i \in \{1, 2, \cdots, p\} \setminus \{s_0, s_1, \cdots, s_d\}} t^{\ell'_{s_j i}} -  \prod_{i \in \{1, 2, \cdots, p\} \setminus \{s_0, s_1, \cdots, s_d\}} t^{-\ell'_{s_j i}} \right)  \ne 0.
\]
Therefore it suffices to show that
\[
\nabla_{\check{L}'} (t, t, \cdots, t)  \ne 0
\]
for some real number $t \in {\mathbb R}$.

Suppose that $\check{L}_1' = L_{s_0}' \cup \check{L}'$. Then $\check{L}_1'$ is a link of $(p-d)$ components with a linking matrix  $\left( \ell_{ij}' \right)$ which satisfies
$\ell_{ij}' \in \{1, -1\}$ for $i\ne j$ and $\ell_{ii}' > 0$ for each $i$.
Let us compare the Hosokawa polynomials of $\check{L}_1'$ and  $\check{L}'$.
Then, by the Torres' condition (\ref{Lemma:Torres}), we get
\[
\nabla_{\check{L}_1'} (1, t, \cdots, t) = - (t^{\ell'} - t^{-\ell'})
\nabla_{\check{L}'} (t, \cdots, t)
\]
for some integer $\ell' = \ell'_{11} >0$ and, by the definition of Hosokawa polynomial,
we have
\begin{align}
\label{Eqn:Hosokawa-2}
 \frac{\nabla_{\check{L}_1'}(1, t, \cdots, t)}{ (t - t^{-1})^{p-d-2}}
 & = -\frac{(t^{\ell'} - t^{-\ell'}) \nabla_{\check{L}'}(t, \cdots , t)}{(t - t^{-1})^{p-d-2}} \\
 &= -(t^{\ell' -1} + t^{\ell' -2} + \cdots + t^{-\ell' +1})   \frac{\nabla_{\check{L}'}(t, \cdots , t)}{(t - t^{-1})^{p-d-3}} \notag\\
&= -(t^{\ell' -1} + t^{\ell' -2} + \cdots + t^{-\ell' +1}) H_{\check{L}'} (t).  \notag
\end{align}
Now, by sending $t$ to $1$, the left hand side of (\ref{Eqn:Hosokawa-2}) becomes
\[
\lim_{t\to 1} \frac{\nabla_{\check{L}_1'}(1, t, \cdots, t)}{ (t - t^{-1})^{p-d-2}}  = H_{\check{L}_1'}(1).
\]
Hence, by Lemma~\ref{Lemma:Hoste}, we get
\[
H_{\check{L}_1'} (1) = (-1)^{p-d-1} \sum_{g\in \mathcal{T}} \overline{g},
\]
where $\overline{g} \in \{1, -1 \}$ because each linking number is $1$ or $-1$. Note that the number of trees in the weighted adjacent graph of $\check{L}_1'$ is $(p-d)^{p-d-2}$  by the Cayley Theorem (\ref{Lemma:Cayley}). Since $p$ is an odd integer and $d$ is an even integer,  $(p-d)^{p-d-2}$  is an odd integer and therefore  $H_{\check{L}_1'}(1) \ne 0$.  Hence it implies
\[
0 \ne H_{\check{L}_1'}(1)  = -\lim_{t\to 1} (t^{\ell' -1} + t^{\ell' -2} + \cdots + t^{-\ell' +1}) H_{\check{L}'} (t) = -(2\ell' -1) H_{\check{L}'}(1).
\]
That is, we have
\[
H_{\check{L}'}(1)  \ne 0.
\]
If we consider the Taylor expansion of $H_{\check{L}'}(t)$ at $t=1$ as in \cite{Buryak:2011}, then
\begin{align*}
\nabla_{\check{L}'} (t, \cdots,t) & = (t-t^{-1})^{p-d -3} H_{\check{L}'}( t)  \\
&= (t-t^{-1})^{p-d -3} \{ H_{\check{L}'}(1)  + \mathcal{O}(t-1) \} \ne 0
\end{align*}
in  some open deleted neighborhood of $1 \in {\mathbb R}$. Therefore
\[
\nabla_{L'} (t_1, \cdots, t_p) \vert_{t_{s_0} =  t_{s_1} = \cdots = t_{s_d} =1  \atop t_j = t \text{ for } j \in \{1, 2, \cdots, p\} \setminus \{s_0, s_1, \cdots s_d\} }   \ne 0
\]
for any choice of \( 1 = s_0 < s_1 < \cdots < s_d \le p \).
Hence we conclude that
\[
\nabla_L (t_1, t_2 \cdots, t_p) \ne \nabla_{L'}(t_{\sigma(1)}, t_{\sigma(2)}, \cdots, t_{\sigma(p)})
\]
for any choice of $\sigma \in S_p$, the permutation group of $\{ 1, 2, \cdots, p\}$.
\end{proof}

\smallskip

\begin{thm}
For each integer $n>0$, the following two symplectic $4$-manifolds
\[
\{ X_{b((2n+1)(4n+1)(4n+3), -(4n+1)(4n+3) + \varepsilon 2(2n+1))}  \ | \ \varepsilon = 1 \text{ or } -1\}
\]
are nondiffeomorphic, but they have the same Seiberg-Witten invariant.
\end{thm}

\begin{proof}
Let us consider the corresponding link surgery $4$-manifolds
\[
\{ \widehat{X}_{b((2n+1)(4n+1)(4n+3), -(4n+1)(4n+3) + \varepsilon 2(2n+1))}  \ | \ \varepsilon = 1 \text{ or } -1\}.
\]
Since the corresponding two dihedral covering links have different diagonal elements by Corollary~\ref{Lemma:diag-Torus} above, they have different multivariable Alexander polynomials by Lemma~\ref{Lemma:Main}, so that they are nondiffeomorphic. Hence we get the result by a covering argument.
\end{proof}

\smallskip

\begin{rem}
Note that $n=1$ case is the Fintushel-Stern's example in~\cite{FS:99}.
\end{rem}

Finally we get our main result by using a similar argument.

\begin{thm}
For each integer $n>0$,  at least $\left[\frac{n+1}{2}\right] + 1$ symplectic $4$-manifolfds in the following families
\[
\{ X_{b(p(n), q(n, i))}  \ | \ i = 0,1, 2, \cdots, 2^n -1\}
\]
are mutually nondiffeomorphic, but they all have the same Seiberg-Witten invariant.
\end{thm}

\begin{proof}
Let $d(n,i)$ be a diagonal element of the corresponding dihedral covering link of a $2$-bridge knot $b(p(n), q(n, i))$. Then Corollary~\ref{Lemma:diag-ni} above implies that
\[
| \{ d(n, i) | d(n, i) \ge 0 \} | = \left[\frac{n+1}{2}\right] + 1.
\]
Therefore, by Lemma~\ref{Lemma:Main}, we can get at least $\left[\frac{n+1}{2}\right] + 1$ inequivalent dihedral covering links which have different multivariable Alexander polynomials. So at least $\left[\frac{n+1}{2}\right] + 1$ link surgery $4$-manifolds in the following families
\[
\{ \widehat{X}_{b(p(n), q(n, i))}  \ | \ i = 0,1, 2, \cdots, 2^n -1\}.
\]
are mutually nondiffeomorphic.  Hence we get a desired result by a covering argument again.
\end{proof}

\medskip

%\bibliographystyle{amsalpha}
%\bibliography{NonDiffeom-20130129}

\providecommand{\bysame}{\leavevmode\hbox to3em{\hrulefill}\thinspace}
\providecommand{\MR}{\relax\ifhmode\unskip\space\fi MR }
\providecommand{\MRhref}[2]{%
  \href{http://www.ams.org/mathscinet-getitem?mr=#1}{#2}
}
\providecommand{\href}[2]{#2}

\end{document}